\documentclass[11pt]{article}
\usepackage{amsmath,mathrsfs}
\usepackage{amssymb}
\usepackage{amsthm}
\usepackage{mathrsfs}
\usepackage[all]{xy}
\usepackage{indentfirst}
\usepackage{color}
\allowdisplaybreaks

\newtheorem{theorem}{Theorem}[section]
\newtheorem*{theorem*}{Theorem}

\newtheorem{remark}[theorem]{Remark}
\newtheorem{lemma}[theorem]{Lemma}
\newtheorem{definition}[theorem]{Definition}

\newtheorem{proposition}[theorem]{Proposition}

\newtheorem*{proposition*}{Proposition}

\newcommand{\R}{\mathbb{R}}

\newcommand{\be}{\begin{eqnarray*}}
\newcommand{\ee}{\end{eqnarray*}}
\newcommand{\ba}{\begin{align*}}
\newcommand{\bpm}{\begin{pmatrix}}
\newcommand{\epm}{\end{pmatrix}}

\newcommand{\bx}{\boldsymbol{x}}

\begin{document}
\title{Almansi-type decomposition and Fueter-Sce theorem for generalized partial-slice regular functions}
\author
{Qinghai Huo$^1$\thanks{This work was partially supported by the National Natural Science Foundation of China (No. 12301097).}
\ Pan Lian$^2$\thanks{This work was partially supported by the National Natural Science Foundation of China (No. 12101451).}
\ Jiajia Si$^3$\thanks{This work was partially supported by the National Natural Science Foundation of China (No. 12201159) and the Hainan Provincial Natural Science Foundation of China (No. 124YXQN413).}
\ Zhenghua Xu$^1$\thanks{This work was partially supported by  the Anhui Provincial Natural Science Foundation (No. 2308085MA04).} \\
\emph{ \small $^1$School of Mathematics, Hefei University of Technology,}\\
\emph{\small  Hefei, 230601, P.R. China}
\\
\emph{\small E-mail:  hqh86@mail.ustc.edu.cn;  zhxu@hfut.edu.cn }\\
\emph{\small $^2$School of Mathematical Sciences, Tianjin Normal University, }\\
\emph{\small  Tianjin 300387,  P.R. China}
\\
\emph{\small E-mail:  panlian@tjnu.edu.cn }\\
\emph{\small $^3$School of Mathematics and Statistics, Hainan University, }\\
\emph{\small Haikou 570228,  P.R. China} \\
\emph{\small E-mail:   sijiajia@mail.ustc.edu.cn  }
}

\maketitle

\begin{abstract}
Very  recently, the concept of generalized partial-slice monogenic (or regular) functions has been introduced to unify  the theory of monogenic  functions and of slice monogenic functions over Clifford algebras.  Inspired by the work
of A. Perotti,    in this paper we provide two analogous versions of the Almansi decomposition in this new setting.  Additionally,  two enhancements of the Fueter-Sce theorem have been obtained for generalized partial-slice regular functions.

\end{abstract}
\section{Introduction}
 In 1898, Emilio Almansi stated  the following decomposition theorem for polyharmonic  functions in terms of harmonic functions (see e.g. \cite{Almansi,Aronszajn}).
\begin{theorem}{\bf(Almansi decomposition)}\label{Almansi-decomposition-harmonic}
 Let $\Omega$ be  a star-like domain in $\mathbb{R}^{n+1}$ with centre $0$ and $u$ be  polyharmonic
 of degree  $N$ in  $\Omega$, that is $\Delta^{N}u(x)=0$, where $\Delta=\sum _{i=0}^{n} \frac{\partial^{2} }{\partial x_{i}^{2}}$ is the Laplacian in   $\mathbb{R}^{n+1}$. Then  there exist unique  harmonic functions  $u_{0},u_{1}, \ldots,u_{N-1}$ in $\Omega$ such that
  $$u(x)=u_{0}(x)+|x|^{2}u_{1}(x)+ \cdots+|x|^{2N-2}u_{N-1}(x), \quad x\in \Omega.$$
\end{theorem}

In  the classical Clifford analysis, we always study  the so-called monogenic functions, which are nullsolutions of the generalized
Cauchy-Riemann systems such as  Dirac operator and Weyl operator.  It is well-known that monogenic functions belong to   harmonic functions with values in Clifford algebras. In 2002, Malonek and Ren first established an Almansi-type theorem in Clifford analysis as follows \cite{Malonek-Ren}.
\begin{theorem} \label{Almansi-decomposition-monogenic}
 Let $\underline{\Omega}$ be  a star-like domain in $\mathbb{R}^{n}$ with centre $0$ and $u$ be left polymonogenic
 of degree $N$ in  $\Omega$, that is $\underline{D}^{N}u(\underline{x})=0$, where the variable $\underline{x}$ is identified as the $1$-vector $\sum _{i=1}^{n}e_i   x_{i}$ in the Clifford algebra $\mathbb{R}_n$ and $\underline{D}=\sum _{i=1}^{n}e_i \frac{\partial }{\partial x_{i}}$ is the classical Dirac operator in $\mathbb{R}^{n}$. Then  there exist unique left monogenic functions  $u_{0},u_{1}, \ldots,u_{N-1}$ in $\underline{\Omega}$ such that
  $$u(\underline{x})=u_{0}(\underline{x})+\underline{x}u_{1}(\underline{x})+ \cdots+\underline{x}^{N-1}u_{N-1}(\underline{x}), \quad \underline{x}\in \underline{\Omega}.$$
\end{theorem}
From the facts that
$$\underline{D}^{2}=- \sum _{i=1}^{n} \frac{\partial^{2} }{\partial x_{i}^{2}}, \ \ \underline{x}^{2}=-|\underline{x}|^{2} =- \sum _{i=1}^{n} | x_{i}|^{2},$$
Theorem  \ref{Almansi-decomposition-monogenic} can be viewed as an refinement of Theorem \ref{Almansi-decomposition-harmonic}.

 Thereafter,   the  Almansi  decomposition has been considered   for different operators, such as \cite{Ren-07}  for   the iterates of weighted Laplace and
 Helmholtz operators,    \cite{Ren-05} for Dunkl operators. Being far from completeness, we refer to \cite{Dinh-22,Ren-06,Ren-Zhang,Yuan} for  more generalizations of  the  Almansi  decomposition. The interested readers can also look up a modified Almansi
 decomposition  for   polyharmonic functions, called the cellular decomposition; see \cite[Theorem 3.4]{Borichev-Hedenmalm} and \cite[Theorem 1.3]{Liu}.

 In the classical Clifford analysis, it is well known that the powers of the
paravector variable fail to be  monogenic.  In order to study the function class including  the powers of paravectors variables, the notions of slice regular  \cite{Gentili-Struppa-07} and slice monogenic  functions \cite{ Colombo-Sabadini-Struppa-09} appear. In the framework of slice analysis, the Almansi-type   theorem  has been considered  \cite{Binosi,Perotti-20} for quaternions and \cite{Perotti-21} for Clifford algebras.

To state Perotti's result from \cite{Perotti-21}, let us  recall the concept of  zonal harmonic functions (cf. \cite{Axler}). Denote by $\mathbb{S}^{n}$ the unit sphere and by $O(n+1)$ the set of orthogonal transformations of $\mathbb{R}^{n+1}$. Given $\eta\in \mathbb{S}^{n}$, consider the harmonic function $Z(\cdot, \eta)$ in $\mathbb{R}^{n+1}$, which is called   the zonal harmonic function with pole $\eta$ if
\begin{equation}\label{zonal}
Z(T(\cdot), \eta)=Z(\cdot, \eta)
\end{equation}
for all $ T\in O(n+1)$ with $T\eta=\eta$.

Now  we can present the Almansi-type decomposition from  Perotti \cite{Perotti-21}  for  slice regular functions.
 \begin{theorem} \label{Almansi-theorem-Perotti-1}
 Let $n=2m+1\geq3$ be odd    and $f$ be slice regular in an axial symmetric  domain $\Omega$ in $\mathbb{R}^{n+1}$.
Then there exist two unique  $\mathbb{R}_{n}$-valued zonal polyharmonic functions  of degree $m$ with pole $1$ in $\Omega$ such that
 $$f(x)=A(x)- \overline{x}B(x).$$
Conversely, if $A$ and $ B$  are $\mathbb{R}_{n}$-valued functions of $C^{1}$ class in $\Omega$, axially symmetric with respect to the real axis, then $g(x)=A( x)- \overline{ x}B( x)$ is a slice function in $\Omega$. The function $g$ is  slice regular   if and only if $A$ and $B$ satisfy the system of equations
 \begin{eqnarray*}
 \left\{
\begin{array}{ll}
 \partial_{x_0} A -  x_0\partial_{x_0}B -r\partial_{r}B =2B,
\\
 \partial_{r}A  -   x_0  B +r \partial_{r}\partial_{x_0} B=0,
\end{array}
\right.
\end{eqnarray*}
where $x=x_0+\underline{x}$, $r=|\underline{x}|$, and $\overline{x}=x_0-\underline{x}$ is the  Clifford conjugation of $x$.
\end{theorem}

From Theorems \ref{Almansi-decomposition-harmonic} and \ref{Almansi-theorem-Perotti-1}, Perotti established  another  decomposition theorem for  slice regular functions  \cite{Perotti-21} .
\begin{theorem} \label{Almansi-theorem-Perotti-2}
 Let $n=2m+1\geq3$ be odd, $\Omega$ be    an axial symmetric star-like domain in $\mathbb{R}^{n+1}$ with centre $0$ and $f$ be  slice regular
 in  $\Omega$. Then  there exist  zonal biharmonic
 functions    $g_{0},g_{1}, \ldots,g_{m-1}$ in $\Omega$, with pole $1$,  such that
  $$f(x)=g_{0}(x)+|x|^{2}g_{1}(x)+ \cdots+|x|^{2m-2}g_{m-1}(x), \quad x\in \Omega,$$
where $g_{0},g_{1}, \ldots,g_{m-1}  \in Ker (D \Delta )$  with  $D=\partial_{x_{0}}+\underline{D}$ and $\Delta$ being the Dirac operator and the Laplacian in   $\mathbb{R}^{n+1}$, respectively.
\end{theorem}

In 2023, the concept of generalized partial-slice monogenic (or regular) functions has been introduced originally to unify  the theory of monogenic  functions and of slice monogenic functions over Clifford algebras; see  \cite{Xu-Sabadini} and the follow up \cite{Ding-Xu}. See   \cite{Xu-Sabadini-2,Xu-Sabadini-3,Xu-Sabadini-4} for more subsequent results.  This idea of partial-slice was carried out and investigated by Ghiloni and Stoppato for quaternions   \cite{Ghiloni-Stoppato-24-1} and developed into the
 theory of $T$-regular functions   in 2024 for an alternative $\ast$-algebra \cite{Ghiloni-Stoppato-24-2}, which proposes a unified theory of regularity in one hypercomplex variable.  In this paper,  we  continue focusing  on generalized partial-slice regular (or monogenic) functions and  provide an analogous version of the Almansi decomposition in this new setting.  Please see Theorems  \ref{Almansi-theorem-Main} and \ref{Almansi-theorem-Perotti-Xu-2},  which generalize   Theorems \ref{Almansi-theorem-Perotti-1} and \ref{Almansi-theorem-Perotti-2}, respectively. To state and prove these  results, in  Section 2 we recall some preliminary notions of generalized partial-slice regular (or monogenic) functions on Clifford algebras. As a consequence of the proof  in  main results,  some enhancements of the Fueter-Sce theorem for generalized partial-slice monogenic functions in \cite{Xu-Sabadini-2} are also obtained; see Theorems \ref{Fueter-Sce-1} and \ref{Futer-Sce-Vekua} in Section 3.

Enclosing the introduction, we point out that  Perotti  in \cite{Perotti-21} established Theorem \ref{Almansi-theorem-Perotti-1}  by using  a very crucial fact that, for any  slice regular $f$,  $x f(x)$ is still  slice regular and $f$ takes the decomposition
$$f(x)=(x f)_{s}^{\prime}(x)- \overline{x}  f_{s}^{\prime}(x),$$
where $f^{\prime}_{s}$ denotes the  spherical derivative of $f$. However,  the situation is different in the more general setting of generalized partial-slice regular (or monogenic) functions, and thus  the  Perotti's proof of Theorem  \ref{Almansi-theorem-Perotti-1} in \cite{Perotti-21} for the slice regular  case should be modified when dealing with Theorem  \ref{Almansi-theorem-Main}.

\section{Preliminaries}
In this section, we recall briefly   some  definitions and preliminary results on    generalized partial-slice monogenic (and regular)
functions over Clifford algebras. See \cite{Xu-Sabadini,Xu-Sabadini-2}.
\subsection{Clifford algebras}
Let $\mathfrak{B}=\{e_1,e_2,\ldots, e_n\}$ be the standard orthonormal basis for   $\mathbb{R}^n$.  Denoted by $\mathbb{R}_{n}$ the real Clifford algebra, which is generated  by $\mathfrak{B}$ satisfying that   $e_i e_j + e_j e_i= -2\delta_{ij}, 1\leq i,j\leq n,$
where $\delta_{ij}$ is the Kronecker symbol.
Every element in  $\mathbb{R}_{n}$ can be written as
 $$a=\sum_A a_Ae_A, \quad a_A\in \mathbb{R},$$
 where
$$e_A=e_{j_1}e_{j_2}\cdots e_{j_r},$$
and $A=\{j_1,j_2, \cdots, j_r\}\subseteq\{1, 2, \cdots, n\}$ with $1\leq j_1< j_2 < \cdots < j_r \leq n,e_\emptyset=e_0=1$.
The norm of $a$ is defined by $|a|= ({\sum_{A}|a_{A}|^{2}} )^{\frac{1}{2}}.$ The Clifford conjugation of $a$ is  an automorphism defined by
$$\overline{a} =\sum_Aa_A\overline{e_A},\quad \overline{e_A}=\overline{e_{j_r}}\ldots\overline{e_{j_1}},$$
 where $ \overline{e_j}=-e_j,1\leq j\leq n,\ \overline{e_0}=e_0=1$.

\subsection{Generalized partial-slice monogenic functions}

Given $n\in \mathbb{N}$,  as usual,  the real space  $\mathbb{R}^{n+1}$ shall be identified with the so-called paravectors of $\mathbb{R}_{n}$  via
$$\mathbb{R}^{n+1}\ni(x_0,x_1,\ldots,x_n) \longmapsto   x=x_{0}+\underline{x}=\sum_{i=0}^{n}e_ix_i.$$
In the sequel, we consider the spitting  $n=p+q,$ where $p$ and $q$ are  nonnegative and  positive integers, respectively. Meanwhile, an element  $x=\bx \in\R^{n+1}$ shall be split into
$$\bx=\bx_p+\underline{\bx}_q,  \quad \bx_p=x_{0}+\underline{\bx}_{p}=\sum_{i=0}^{p}x_i e_i\in\R^{p+1},\ \underline{\bx}_q=\sum_{i=p+1}^{p+q}x_i e_i\in\R^q.$$
Note that   here  and all the rest we write the variable $x$ as $\bx$ in bold to emphasize the splitting.

Similarly, the Dirac operator $D_{\bx}$ is split as
\begin{equation*}
D_{\bx}=D_{\bx_p} +D_{\underline{\bx}_q}, \quad D_{\bx_p} =\sum_{i=0}^{p}e_i\partial_{x_i}, \
D_{\underline{\bx}_q} =\sum_{i=p+1}^{p+q}e_i\partial_{x_i}.
\end{equation*}

Denote by $\mathbb{S}$ the sphere of unit $1$-vectors in $\mathbb R^q$, i.e.
$$\mathbb{S}=\big\{\underline{\bx}_q: \underline{\bx}_q^2 =-1\big\}=\big\{\underline{\bx}_q=\sum_{i=p+1}^{p+q}x_i e_i:\sum_{i=p+1}^{p+q}x_i^{2}=1\big\}.$$
Note that, for $\underline{\bx}_q\neq0$, there exists   uniquely  $r>0$ and $\underline{\omega}\in \mathbb{S}$, such that $\underline{\bx}_q=r\underline{\omega}$, more precisely
 $$r=|\underline{\bx}_q|, \quad \underline{\omega}=\frac{\underline{\bx}_q}{|\underline{\bx}_q|}. $$
 When $\underline{\bx}_q= 0$,  $r=0$ but $\underline{\omega}$ is not uniquely, due to that    $\bx=\bx_p+\underline{\omega} \cdot 0$ for all $\underline{\omega}\in \mathbb{S}$.

Now we recall  the  notion of  generalized partial-slice monogenic functions \cite{Xu-Sabadini}.
\begin{definition} \label{definition-slice-monogenic}
 Let $\Omega$ be a domain in $\mathbb{R}^{p+q+1}$. A function $f :\Omega \rightarrow \mathbb{R}_{p+q}$ is called left   \textit{generalized partial-slice monogenic} of type $(p,q)$ if, for all $ \underline{\omega} \in \mathbb S$, its restriction $f_{\underline{\omega}}$ to $\Omega_{\underline{\omega}}:=\Omega\cap (\mathbb{R}^{p+1} \oplus \underline{\omega} \mathbb{R}) \subseteq \mathbb{R}^{p+2}$  has continuous partial derivatives and  satisfies
$$D_{\underline{\omega}}f_{\underline{\omega}}(\bx):=(D_{\bx_p}+\underline{\omega}\partial_{r}) f_{\underline{\omega}}(\bx_p+r\underline{\omega})=0,$$
for all $\bx=\bx_p+r\underline{\omega} \in \Omega_{\underline{\omega}}$.
 \end{definition}
Throughout this article,  we always deal with left generalized partial-slice monogenic functions  of type $(p,q)$ and  hence omit to specify type $(p,q)$ and  denote it by $\mathcal {GSM}(\Omega)$ (or $\mathcal {GSM}^{L}(\Omega)$ when needed).

Likewise, denote by $\mathcal {GSM}^{R}(\Omega)$ the left Clifford module of  all right  generalized partial-slice monogenic functions of type $(p,q)$ $f:\Omega  \rightarrow \mathbb{R}_{p+q}$ which are defined by requiring that
$$f_{\underline{\omega}}(\bx)D_{\underline{\omega}}:={f_{\underline{\omega}} (\bx_p+r\underline{\omega})D_{\bx_p}}+ \partial_{r}f_{\underline{\omega}} (\bx_p+r\underline{\omega})\underline{\omega}=0, \quad \bx=\bx_p+r\underline{\omega} \in \Omega_{\underline{\omega}}.$$

\begin{remark} \rm{
 When $(p,q)=(n-1,1)$, the  notion of generalized partial-slice monogenic functions in Definition  \ref{definition-slice-monogenic} coincides with \textit{monogenic functions}  defined in $\Omega\subseteq\mathbb{R}^{n+1}$   with values in the Clifford algebra $\mathbb{R}_{n}$. For more details on the theory of monogenic functions, see e.g. \cite{Brackx,Colombo-Sabadini-Sommen-Struppa-04,Delanghe-Sommen-Soucek,Gurlebeck}.}
\end{remark}
\begin{remark}\rm{
When $(p,q)=(0,n)$, Definition  \ref{definition-slice-monogenic}  boils down to  \textit{slice monogenic functions} defined in $\Omega\subseteq\mathbb{R}^{n+1}$ and   with values in the Clifford algebra $\mathbb{R}_{n}$; see \cite{Colombo-Sabadini-Struppa-09,Colombo-Sabadini-Struppa-11}. }
\end{remark}


\begin{definition} \label{slice-domain}
 Let $\Omega$ be a domain in $\mathbb{R}^{p+q+1}$.

1.   $\Omega$ is called  slice domain if $\Omega\cap\mathbb R^{p+1}\neq\emptyset$  and $\Omega_{\underline{\omega}}$ is a domain in $\mathbb{R}^{p+2}$ for every  $\underline{\omega}\in \mathbb{S}$.

2.   $\Omega$   is called  partially  symmetric with respect to $\mathbb R^{p+1}$ (p-symmetric for short) if, for   $\bx_{p}\in\R^{p+1}, r \in \mathbb R^{+},$ and $ \underline{\omega}  \in \mathbb S$,
$$\bx=\bx_p+r\underline{\omega} \in \Omega\Longrightarrow [\bx]:=\bx_p+r \mathbb S=\{\bx_p+r \underline{\omega}, \ \  \underline{\omega}\in \mathbb S\} \subseteq \Omega. $$
 \end{definition}
Generalized partial-slice monogenic  functions defined on p-symmetric slice domain   possess  the representation formula.
\begin{theorem}  {\bf(Representation Formula)}  \label{Representation-Formula-SM}
Let $\Omega\subseteq \mathbb{R}^{p+q+1}$ be a p-symmetric slice domain and $f:\Omega\rightarrow \mathbb{R}_{p+q}$ be a  generalized partial-slice monogenic function.  Then it holds that,  for $\bx_p+r\underline{\omega} \in \Omega$,
\begin{equation*}\label{Representation-Formula-eq}
f(\bx_p+r \underline{\omega})=\frac{1}{2} (f(\bx_p+r\underline{\eta} )+f(\bx_p-r\underline{\eta}) )+
\frac{ 1}{2} \underline{\omega}\underline{\eta} (  f(\bx_p-r\underline{\eta} )-f(\bx_p+r\underline{\eta})),
\end{equation*}
for any $\underline{\eta}\in \mathbb{S}$.

Moreover, the following two functions do not depend on $\underline{\eta}$:
$$F_1(\bx_p,r)=\frac{1}{2} (f(\bx_p+r\underline{\eta} )+f(\bx_p-r\underline{\eta} ) ),$$
$$F_2(\bx_p,r)=\frac{ 1}{2}\underline{\eta}(  f(\bx_p-r\underline{\eta} )-f(\bx_p+r\underline{\eta})).$$
\end{theorem}

\subsection{Generalized partial-slice  functions}
Throughout this article, let $D \subseteq \mathbb{R}^{p+2}$  be a domain, which is invariant under the reflection  of the $(p+2)$-th variable, i.e.
$$ \bx':=(\bx_p,r) \in D \Longrightarrow   \bx_\diamond':=(\bx_p,-r)  \in D.$$
The  \textit{partially symmetric completion} $ \Omega_{D}\subseteq\mathbb{R}^{p+q+1}$ of  $D$ is defined by
$$\Omega_{D}=\bigcup_{\underline{\omega} \in \mathbb{S}} \, \big \{\bx_p+r\underline{\omega}\  : \ \exists \ \bx_p \in \mathbb{R}^{p+1},\ \exists \ r\geq 0,\  \mathrm{s.t.} \ (\bx_p,r)\in D \big\}.$$

By Definition  \ref{slice-domain}, it is easy to see that a  domain $\Omega \subseteq \mathbb{R}^{p+q+1}$ is p-symmetric if and only if $\Omega=\Omega_{D}$ for some domain $D \subseteq\mathbb{R}^{p+2}$.
\begin{definition}
Let $D\subseteq \mathbb{R}^{p+2}$ be  a domain, invariant under the reflection  of the $(p+2)$-th variable.
A function $f=\mathcal I(F): \Omega_D\longrightarrow  \mathbb{R}_{p+q}$ of the form
  \begin{equation*}\label{genpslice}
  f(\bx)=F_1(\bx')+\underline{\omega} F_2(\bx'), \qquad   \bx=\bx_p+r\underline{\omega}   \in \Omega_{D},
  \end{equation*}
 where the $\mathbb{R}_{p+q} $-valued components $F_1, F_2$ of  $F=(F_1, F_2)$  satisfy
 \begin{equation*}\label{even-odd}
 F_1(\bx_{\diamond}')= F_1(\bx'), \qquad  F_2(\bx_{\diamond}')=-F_2(\bx'), \qquad  \bx' \in D,  \end{equation*}
is called a (left)  generalized partial-slice function.
\end{definition}

Denote by $\mathcal{GS}(\Omega_{D})$ the set of  all  generalized partial-slice functions  on $\Omega_{D}$ induced  by $(F_1,F_2)$  and by ${\mathcal{GS}}^{k}(\Omega_{D})$ when the components $F_1,F_2$ are of class $C^k(D)$.
\begin{definition}\label{definition-GSR}
Let $f(\bx)=F_1(\bx')+\underline{\omega} F_2(\bx') \in {\mathcal{GS}}^{1}(\Omega_{D})$. The function $f$ is called generalized partial-slice regular   of type $(p,q)$ if  $F_1, F_2$ satisfy  the generalized Cauchy-Riemann equations
 \begin{eqnarray}\label{C-R}
 \left\{
\begin{array}{ll}
D_{\bx_p}  F_1(\bx')- \partial_{r} F_2(\bx')=0,
\\
 \overline{D}_{\bx_p}  F_2(\bx')+ \partial_{r} F_1(\bx')=0,
\end{array}
\right.
\end{eqnarray}
for all $\bx'\in D$.
\end{definition}
As before, the type $(p,q)$ will be omitted in the sequel. Denote by $\mathcal {GSR}(\Omega_D)$ the set of all generalized partial-slice  regular  functions {on $\Omega_D$}.

From \cite[Proposition 3.5]{Xu-Sabadini-2} and   the proof of \cite[Lemma 3.7]{Xu-Sabadini-2}, we have
\begin{lemma}  \label{GSR-Harmonic}
  Let $f(\bx)=F_1(\bx')+\underline{\omega} F_2(\bx') \in  \mathcal{GSR}  (\Omega_{D})$. Then $F_1, F_2$  are   real analytic on $D$, $f$ is real analytic on $\Omega_{D}$ and
    \[\Delta_{\bx'}F_1 =0,\quad \Delta_{\bx'}F_2 =0,\]
  where $\Delta_{\bx'}$ is the Laplacian  in $\mathbb {R} ^{p+2}$.
\end{lemma}

Let us recall \cite[Lemma 3.4]{Xu-Sabadini-2}:
\begin{lemma}\label{Fueter-Sce-Qian-lemma-1}
Let    $f(\bx)=F_1(\bx')+\underline{\omega} F_2(\bx') \in {\mathcal{GS}}^{2}(\Omega_{D})\cap C^{2}(\Omega_{D})$. Then it holds that for all $\bx \in\Omega_{D}$
\begin{eqnarray*}\label{relation-harm}
\Delta_{\bx} f(\bx)= \Delta_{\bx'}F_1(\bx')+ \underline{\omega} \Delta_{\bx'}F_2(\bx')+ (q-1) ( (\frac{1}{r}\partial_{r}) F_1(\bx')+\underline{\omega} (\partial_{r}\frac{1}{r}) F_2(\bx') ).
\end{eqnarray*}
\end{lemma}
Due to Lemmas \ref{GSR-Harmonic} and \ref{Fueter-Sce-Qian-lemma-1}, one can  obtain directly   that,  for $f(\bx)=F_1(\bx')+\underline{\omega} F_2(\bx') \in  \mathcal{GSR}  (\Omega_{D})$, the part $F_1(\bx')$ is hyperbolic harmonic
$$\big(r^{2}\Delta_{\bx}-(q-1)r \partial_{r}  \Big) F_1(\bx')=0, $$
and $\underline{\omega}F_2(\bx')$  is the eigenfunction of the hyperbolic Laplace operator
$$\big(r^{2}\Delta_{\bx}-(q-1)r \partial_{r}  \Big) (\underline{\omega}F_2(\bx'))=-(q-1)\underline{\omega}F_2(\bx'). $$

Now we recall a relationship between the set of functions  $\mathcal {GSM}$ and $\mathcal {GSR}$ defined in  p-symmetric domains  \cite{Xu-Sabadini}.

\begin{theorem} \label{relation-GSR-GSM}

(i) For a p-symmetric domain $\Omega=\Omega_{D}$ with $\Omega  \cap \mathbb{R}^{p+1}= \emptyset$, it holds that $\mathcal {GSM}(\Omega) \supsetneqq \mathcal {GSR}(\Omega_{D})$.

(ii) For a p-symmetric domain $\Omega=\Omega_{D}$ with $\Omega  \cap \mathbb{R}^{p+1}\neq \emptyset$,  it holds  that $\mathcal {GSM}(\Omega) = \mathcal {GSR}(\Omega_{D})$.
\end{theorem}

\begin{definition}\label{GSM-GSR}
Let $f\in   \mathcal{GS} (\Omega_{D}) $.
The function $f^{\circ}_{s}: \Omega_{D} \longrightarrow \mathbb{R}_{p+q}$, called spherical value of $f$, and the   function $f^{\prime}_{s}: \Omega_{D}\setminus \mathbb{R}^{p+1} \longrightarrow \mathbb{R}_{p+q}$,  called spherical derivative of $f$, are defined as, respectively,
$$f^{\circ}_{s}(\bx)=\frac{1}{2}\big( f(\bx)+f(\bx_{\diamond}) \big), $$
$$f^{\prime}_{s}(\bx)=\frac{1}{2}\underline{\bx}_{q}^{-1}\big( f(\bx)-f(\bx_{\diamond}) \big).$$
\end{definition}

By definition, it hold that
$$f^{\circ}_{s}(\bx)= F_1 (\bx')=f^{\circ}_{s}(\bx_{\diamond}), $$
$$f^{\prime}_{s}(\bx)= \frac{1}{r}F_2 (\bx')= f^{\prime}_{s}(\bx_{\diamond}),$$
and \begin{equation}\label{spherical-value-derivative}
 f(\bx)=f^{\circ}_{s}(\bx)+ \underline{\bx}_{q} f_{s}^{\prime}(\bx), \quad \bx \in \Omega_{D}\setminus \mathbb{R}^{p+1}.
\end{equation}

For $\Omega_{D} \cap \mathbb{R}^{p+1}\neq \emptyset$ and  $F_2\in C^{1}({D})$, it follows that
$$\lim_{r\rightarrow0}\frac{1}{r}F_2(\bx_{p},r)=\lim_{r\rightarrow0}\frac{1}{r}(F_2(\bx_{p},r)-F_2(\bx_{p},0))=\partial_{r} F_2(\bx_{p},0),\quad  (\bx_{p},0)\in  D. $$
At this moment,  $f^{\prime}_{s}$ could extend continuously to the whole space $\Omega_{D}$.

Furthermore, thanks to the even-oddness of the real analytic functions pair $(F_1,F_2)$ w.r.t. $r$,  there  exists a pair of real analytic functions $(G_1,G_2)$  such that
$$  F_ 1 ( \boldsymbol { x } _ p, r ) = G _ 1 ( \boldsymbol { x } _ p, r ^ 2 ), \quad F _ 2 ( \boldsymbol { x } _ p, r ) = r G_ 2 ( \boldsymbol { x } _ p, r ^ 2 ). $$
Hence,
\begin{equation}\label{spherical-value-derivative-G}
 f_s^\circ(\boldsymbol x)=G_1(\boldsymbol{x}_p,r^2),\quad  f_{s}^{\prime}(\boldsymbol x)=G_{2}(\boldsymbol{x}_p,r^2).
\end{equation}

\section{Main results}

 To establish the Almansi-type decomposition in our setting,  we shall follow the idea of Perotti in \cite[Theorem 3.4.1]{Perotti-19} by    computing the  Laplacian of spherical derivative and spherical value for   generalized partial-slice    functions.
\begin{lemma}\label{harmonic-Xu-lemma}
Let $K=2(q+1)\in \mathbb{N}, $  and $f(\bx)=F_1(\bx')+\underline{\omega} F_2(\bx') \in {\mathcal{GS}}^{K}(\Omega_{D})$.
Let $f_s^\circ(\boldsymbol x)=G_1(\boldsymbol{x}_p,r^2)$ and $ f_{s}^{\prime}(\boldsymbol x)=G_{2}(\boldsymbol{x}_p,r^2)$ as in (\ref{spherical-value-derivative-G})   and denote $\partial_{p+2}G_j (\boldsymbol{x}_p,r^2)=\frac{\partial G_j}{\partial{r} }(\boldsymbol{x}_p,r^2), $ $j=1,2.$
If $F_j$ satisfy the Helmholtz equation
\begin{equation}\label{Helmholtz-equation}
(\Delta_{\bx'}+\lambda^{2})F_j(\bx') =0,\quad j=1,2,
\end{equation}
where $\Delta_{\bx'}$ is the Laplacian in $\mathbb {R} ^{p+2}$ and  $\lambda  \in \mathbb{R}$ is a constant,   then we have \\
$({i})$
 For each $k=1,2,\ldots, \left[\frac{q-1}{2}\right],$
\[(\Delta_{\bx}+\lambda^{2})^k f_s^{\prime}(\boldsymbol{x} )=2^k(q-3)(q-5)\cdots(q-2k-1)\partial_{p+2} ^k G_2(\boldsymbol{x}_p,r^2).\]
$({ii})$
For each $k=1,2,\ldots,\left[\frac{q+1}{2}\right]$,
\[
(\Delta_{\bx}+\lambda^{2})^kf_s^\circ(\boldsymbol{x} )=2^k(q-1)(q-3)\cdots(q-2k+1)\partial_{p+2}^kG_1(\boldsymbol{x}_p,r^2).\]
\end{lemma}

\begin{proof}
(i) We shall continue to use the symbols from Section $2$ and  write $f=\mathcal{I}(F) \in   \mathcal{GS}^{K} (\Omega_{D})$ as
$$f(\bx)=f^{\circ}_{s}(\bx)+ \underline{\bx}_{q} f_{s}^{'}(\bx), $$
where
\[f^{\circ}_{s}(\bx)=F_1(\boldsymbol{x}_p,r)=G_1(\boldsymbol{x}_p,r^2),\quad f_{s}^{\prime}(\boldsymbol{x} )=\frac{1}{r}F_2(\boldsymbol{x}_p,r)=G_{2}(\boldsymbol{x}_p,r^2).\]
 We shall use mathematical induction on $k$. First, let us  prove the result for $k = 1 $, i.e.
$$(\Delta_{\bx}+\lambda^{2})f_{s}^{\prime}(\boldsymbol{x} )=2(q-3)\partial_{p+2}G_{2}(\boldsymbol{x}_p,r^2).$$
Under the condition in (\ref{Helmholtz-equation}), it follows that
 \[(\Delta_{\bx'}+\lambda^{2})F_2(\boldsymbol{x}_p,r)=r(\Delta_{\bx_{p}}+\lambda^{2})G_2(\boldsymbol{x}_p,r^{2})+\partial_r^2 (rG_2(\boldsymbol{x}_p,r^{2}))=0,\]
which implies
\begin{align}\label{relation-L}
(\Delta_{\bx_{p}}+\lambda^{2}) G_{2}(\boldsymbol{x}_p,r^2)&=-\frac{1}{r} \partial_r^2\left(rG_2(\boldsymbol{x}_p,r^2)\right)\nonumber \\&=\frac{-1}{r} \partial_r\big(G_2(\boldsymbol{x}_p,r^2)+2r^3 \partial_{p+2}G_2(\boldsymbol{x}_p,r^2)\big)\nonumber\\
&=-6\partial_{p+2}G_2(\boldsymbol{x}_p,r^2)-4r^2\partial_{p+2}^2 G_2(\boldsymbol{x}_p,r^2). \end{align}
Hence, by  straightforward computations,
\begin{align}
(\Delta_{\bx}+\lambda^{2}) f_{s}^{\prime}(\boldsymbol{x} )
&=(\Delta_{\bx_{p}}+\lambda^{2})G_{2}(\boldsymbol{x}_p,r^2)+ \Delta_{\underline{\boldsymbol{x}}_q} (G_{2}(\boldsymbol{x}_p,r^2))\nonumber\\
&=(\Delta_{\bx_{p}}+\lambda^{2})G_{2}(\boldsymbol{x}_p,r^2)+\sum _{i={p+1}}^{p+q} \frac{\partial}{\partial{x_i}}\big(2x_i\partial_{p+2}G_2(\boldsymbol{x}_p,r^2)\big)\nonumber\\
&=(\Delta_{\bx_{p}}+\lambda^{2})G_2(\boldsymbol{x}_p,r^2)+4r^2\partial_{p+2}^2 G_{2}(\boldsymbol{x}_p,r^2)+2q\partial_{p+2} G_{2}(\boldsymbol{x}_p,r^2) \nonumber \\
&=2(q-3)\partial_{p+2} G_{2}(\boldsymbol{x}_p,r^2). \nonumber
\end{align}

Now we   suppose that (i) holds for $1<k\leq \left[\frac{q-1}{2}\right]-1$.
 To show (i) holds for $k+1$, we need the following
 \begin{eqnarray*}
& & (\Delta_{\bx}+\lambda^{2}) (\partial_{p+2}^{k}G_{2}(\boldsymbol{x}_{p},r^{2}))\\
&=&(\Delta_{\boldsymbol{x}_p}+\lambda^{2})\partial_{p+2}^{k}G_2(\boldsymbol{x}_{p},r^{2})+4r^{2}\partial_{p+2}^{k+2}
G_2(\boldsymbol{x}_{p},r^{2})+2q\partial_{p+2}^{k+1}G_2(\boldsymbol{x}_{p},r^{2})
\\
 &=&  \partial_{p+2}^{k}\left((\Delta_{\boldsymbol{x}_p}+\lambda^{2})G_2(\boldsymbol{x}_{p},r^{2})+4r^{2}\partial_{p+2}^{2}G_2(\boldsymbol{x}_{p},r^{2})
 +2q\partial_{p+2}G_{2}(\boldsymbol{x}_{p},r^{2})\right) \\
 && -4k\partial_{p+2}^{k+1}G_2(\boldsymbol{x}_{p},r^{2})
 \\
 &=& \partial_{p+2}^k\big (\left(-6+2q\right)\partial_{p+2} G_2(\boldsymbol{x}_{p},r^{2}) \big)-4k\partial_{p+2}^{k+1}G_2(\boldsymbol{x}_{p},r^{2})
 \\
&=& 2(q-2k-3)\partial_{p+2}^{k+1} G_2(\boldsymbol{x}_{p},r^{2}),
\end{eqnarray*}
 where the third equality comes from (\ref{relation-L}).

Therefore, by the induction hypothesis, we infer that
\[\begin{aligned}
(\Delta_{\boldsymbol{x}}+\lambda^{2})^{k+1}f_s'(\boldsymbol{x} )
&=2^k(q-3)(q-5)\cdots(q-2k-1)(\Delta_{\boldsymbol{x}}+\lambda^{2}) (\partial_{p+2}^kG_2(\boldsymbol{x}_p,r^2))\\
&=2^{k+1}(q-3)(q-5)\cdots (q-2k-3)\partial_{p+2}^{k+1}G_2(\boldsymbol{x}_p,r^2).\end{aligned}\]

(ii)
In a similar way to the proof in (i), we use mathematical induction.  First, we prove the result for $k=1$, i.e.
$$(\Delta_{\bx}+\lambda^{2})f_{s}^{\circ}(\boldsymbol{x} ) =2(q-1)\partial_{p+2} G_1(\boldsymbol{x}_p,r^2).$$
Under the condition in (\ref{Helmholtz-equation}), it holds that
 \[(\Delta_{\bx'} +\lambda^{2})F_1(\boldsymbol{x}_p,r)= (\Delta_{\boldsymbol{x}_p}+\lambda^{2})G_1(\boldsymbol{x}_p,r^{2})+\partial_r^2 ( G_1(\boldsymbol{x}_p,r^{2}))=0,\]
which implies
\begin{align}\label{relation-LL}
(\Delta_{\boldsymbol{x}_p}+\lambda^{2}) G_{1}(\boldsymbol{x}_p,r^2)&=-  \partial_r^2\left( G_1(\boldsymbol{x}_p,r^2)\right)\nonumber \\
&= -2\partial_{p+2}G_1(\boldsymbol{x}_p,r^2)-4r^2\partial_{p+2}^2 G_1(\boldsymbol{x}_p,r^2).\end{align}
Hence, we get
\[\begin{aligned}
 (\Delta_{\bx}+\lambda^{2}) f_{s}^{\circ}(\boldsymbol{x} )
&=(\Delta_{\bx_{p}}+\lambda^{2})G_1(\boldsymbol{x}_p,r^2)+4r^2\partial_{p+2}^2 G_{1}(\boldsymbol{x}_p,r^2)+2q\partial_{p+2} G_{1}(\boldsymbol{x}_p,r^2)
\\
&= 2(q-1)\partial_{p+2} G_{1}(\boldsymbol{x}_p,r^2).
\end{aligned}\]
Now  suppose that (ii) holds for $1<k\leq \left[\frac{q+1}{2}\right]-1$. Now it remains to show that (ii) holds for $k+1$.  In fact,
 \begin{eqnarray*}
& & (\Delta_{\bx}+\lambda^{2}) (\partial_{p+2}^{k}G_{1}(\boldsymbol{x}_{p},r^{2}))\\
&=& (\Delta_{\boldsymbol{x}_p}+\lambda^{2})\partial_{p+2}^{k}G_1(\boldsymbol{x}_{p},r^{2})+4r^{2}\partial_{p+2}^{k+2}
G_1(\boldsymbol{x}_{p},r^{2})+2q\partial_{p+2}^{k+1}G_1(\boldsymbol{x}_{p},r^{2})
\\
 &=&  \partial_{p+2}^{k}\left((\Delta_{\bx_{p}}+\lambda^{2})G_1(\boldsymbol{x}_{p},r^{2})+4r^{2}\partial_{p+2}^{2}G_1(\boldsymbol{x}_{p},r^{2})
 +2q\partial_{p+2}G_{1}(\boldsymbol{x}_{p},r^{2})\right) \\
 && -4k\partial_{p+2}^{k+1}G_1(\boldsymbol{x}_{p},r^{2})
 \\
 &=& \partial_{p+2}^k\big (\left(2q-2\right)\partial_{p+2} G_1(\boldsymbol{x}_{p},r^{2}) \big)-4k\partial_{p+2}^{k+1}G_1(\boldsymbol{x}_{p},r^{2})
 \\
&=& 2(q-2k-1)\partial_{p+2}^{k+1} G_1(\boldsymbol{x}_{p},r^{2}),
\end{eqnarray*}
 where the third equality comes from (\ref{relation-LL}).

 Therefore, by the induction hypothesis, we infer
\[\begin{aligned}
(\Delta_{\boldsymbol{x}}+\lambda^{2})^{k+1}f_s^{\circ}(\boldsymbol{x} )
&=2^k(q-3)(q-5)\cdots(q-2k+1)(\Delta_{\boldsymbol{x}} +\lambda^{2})(\partial_{p+2}^kG_1(\boldsymbol{x}_p,r^2))\\
&=2^{k+1}(q-3)(q-5)\cdots (q-2k-1) \partial_{p+2}^{k+1}G_1(\boldsymbol{x}_p,r^2).\end{aligned}\]
The proof is complete.
\end{proof}

\begin{remark} \rm{
When the Helmholtz equation in (\ref{Helmholtz-equation}) is replaced by the Klein-Gordon equation
\begin{equation*}
(\Delta_{\bx'}-\lambda^{2})F_j(\bx') =0,\quad j=1,2,
\end{equation*}
the conclusion in Lemma \ref{harmonic-Xu-lemma}  still holds if we consider the operator $(\Delta_{\bx}-\lambda^{2})^{k} $, instead of $(\Delta_{\bx}+\lambda^{2})^{k} $.

In addition,  Lemma \ref{harmonic-Xu-lemma} holds in a more general case  $\lambda=\lambda(\boldsymbol{x}_p) \in \mathbb{R}_n$.

}
\end{remark}

\begin{theorem}\label{spherical-derivative-value}
  Let $D\subseteq\mathbb{R}^{p+2}$ be  a domain, $f=\mathcal{I}(F): \Omega_D \longrightarrow \mathbb {R}_{p+q} $ be a generalized partial-slice regular function. Let $f_s^\circ(\boldsymbol x)=G_1(\boldsymbol{x}_p,r^2)$ and $ f_{s}^{\prime}(\boldsymbol x)=G_{2}(\boldsymbol{x}_p,r^2)$ as in (\ref{spherical-value-derivative-G})   and denote $\partial_{p+2}G_j (\boldsymbol{x}_p,r^2)=\frac{\partial G_j}{\partial{r} }(\boldsymbol{x}_p,r^2), j=1,2.$
 Then we have \\
(i) For each $k=1,2,\ldots, \left[\frac{q-1}{2}\right],$
\[\Delta_{\bx}^k f_s^{\prime}(\boldsymbol{x} )=2^k(q-3)(q-5)\cdots(q-2k-1)\partial_{p+2} ^k G_2(\boldsymbol{x}_p,r^2).\]
(ii)
  For each $k=1,2,\ldots,\left[\frac{q+1}{2}\right]$,
\[\Delta_{\bx}^kf_s^\circ(\boldsymbol{x} )=2^k(q-1)(q-3)\cdots(q-2k+1)\partial_{p+2}^kG_1(\boldsymbol{x}_p,r^2).\]
(iii)
$$\Delta_{\bx}f_{s}^{\prime}(\boldsymbol{x} ) =\frac{q-3}{r^{2}}\left(D_{\boldsymbol{x}_p } f_{s}^{\circ}(\boldsymbol{x} )-f_{s}^{\prime}(\boldsymbol{x} )\right). $$
(vi)
\[\Delta_{\bx} f_s^\circ(\boldsymbol{x} ) =(1-q)\overline{D}_{\boldsymbol x_p} f_s^\prime (\boldsymbol{x}).\]
(v)
If $q$ is odd, $ \Delta_{\bx} ^\frac{q+1}{2} f_s^\circ (\boldsymbol{x} )=0$.
\end{theorem}
\begin{proof}
(i) and (ii) follows directly from by Lemmas \ref{GSR-Harmonic} and \ref{harmonic-Xu-lemma}.

(iii) In view of that $F_1,F_2$ satisfy the generalized Cauchy-Riemann equation in (\ref{C-R}), it holds
$$D_{\boldsymbol{x}_p} F_1(\boldsymbol{x}_{p},r)=\partial_rF_2(\boldsymbol{x}_{p},r)=\partial_r(rG_2(\boldsymbol{x}_p,r^2) ) =G_2(\boldsymbol{x}_p,r^2)+2r^2\partial_{p+2}G_2(\boldsymbol{x}_p,r^2),$$
which gives that
\[\partial_{p+2}G_{2}(\boldsymbol{x}_{p},r^{2})=\frac{1}{2r^{2}}\left(D_{\boldsymbol{x}_p}F_{1}(\boldsymbol{x}_{p},r)-G_{2}(\boldsymbol{x}_{p},r^{2})\right)=\frac{1}{2r^{2}}\left(D_{\boldsymbol{x}_p} f_{s}^{\circ}(\boldsymbol{x} )-f_{s}^{\prime}(\boldsymbol{x} )\right).\]
Combining this    formula with (i), we have
\[ \Delta_{\boldsymbol{x}}f_s^\prime(\boldsymbol{x} )=2(q-3)\partial_{p+2}G_2(\boldsymbol{x}_{p},r^{2})=\frac{q-3}{r^2}\left(D_{\boldsymbol{x}_p} f_{s}^{\circ}(\boldsymbol{x})-f_{s}^{\prime}(\boldsymbol{x} )\right).\]

(iv) In view of that $F_1,F_2$ satisfy the generalized Cauchy-Riemann equation in (\ref{C-R})
$$\overline{D}_{\boldsymbol{x}_p} F_2(\boldsymbol{x}_p,r)=-\partial_r F_1(\boldsymbol{x}_p,r),$$
 we obtain
\[r\overline{D}_{\boldsymbol{x}_p} f_s ^ \prime(\boldsymbol{x} )=-\partial_r (G_1(\boldsymbol{x}_p,r^2) ) =-2r\partial_{p+2} G_1(\boldsymbol{x}_p,r^2),\]
i.e.
\[\partial_{p+2} G_1(\boldsymbol{x}_p,r^2)=-\frac{ 1}{2} \overline{D}_{\boldsymbol{x}_p} f_s ^ \prime(\boldsymbol{x} ).\]
Substituting the above formula into (ii) gives that
\[\Delta_{\bx} f_s^\circ (\boldsymbol{x} )=2(q-1)\partial_{p+2} G_1(\boldsymbol{x}_p,r^2)=(1-q)\overline{D}_{\boldsymbol x_p} f_s^\prime (\boldsymbol{x} ).\]
(v)
For odd $q$, setting $k=\frac{q+1}{2}$  in (ii),  (v) can be obtained.

The proof is complete.
\end{proof}
\begin{remark}\rm{
When $(p,q)=(0,n)$, Theorem \ref{spherical-derivative-value} is consistent with of Theorems 3.4.1  and   3.4.6  in \cite{Perotti-19}  for slice regular functions over the Clifford algebra $\mathbb R_n$.}
\end{remark}

In \cite{Xu-Sabadini-2}, a Fueter-Sce  theorem for generalized partial-slice regular functions was obtained.
\begin{theorem}\label{Fueter-Sce-theorem}
Let     $f(\bx)=F_1(\bx')+\underline{\omega} F_2(\bx') \in \mathcal{GSR}(\Omega_{D})$.  Then, for odd $q\in \mathbb{N}$,
the function
 $$\tau_{q}f(\bx):=\frac{1}{(q-1)!!}\Delta_{\bx}^{\frac{q-1}{2}}f(\bx)$$
  is monogenic in $\Omega_{D}$.
 \end{theorem}

As an application  of Theorem \ref{spherical-derivative-value}, two  generalized versions of the Fueter-Sce theorem can be obtained. To   this end, we recall a technical lemma from \cite{Xu-Sabadini}.
 \begin{proposition}\label{relation}
If $f \in {\mathcal{GS}}^{1}(\Omega_{D}) $, then
 $$(D_{\bx}-{D_{\underline{\omega}})}f(\bx) =(1-q) f_s^\prime (\boldsymbol{x}), \quad \boldsymbol{x}\in\Omega_{D}\setminus \mathbb{R}^{p+1}.$$
\end{proposition}

Now we can   make an  enhancement of  Theorem   \ref{Fueter-Sce-theorem} as follows.
\begin{theorem}\label{Fueter-Sce-1}
Let $q\in \mathbb{N}$ be  odd  and $f\in \mathcal{GSR} (\Omega_{D})$. Then \\
(i)  The generalized Fueter-Sce theorem holds:
 \[ D_{\boldsymbol{x}}  \Delta_{\bx} ^{\frac{q-1}{2}} f(\boldsymbol{x} )=(1-q) \Delta_{\bx} ^{\frac{q-1}{2}} f_{s}^\prime(\boldsymbol{x} )=0.\]
 (ii) $ \Delta_{\bx} ^{\frac{q+1}{2}}f(\boldsymbol{x} )=0$, i.e. every generalized partial-slice regular function  is polyharmonic of degree $\frac{q+1}{2}$.
 \end{theorem}
 \begin{proof}
(i)  Let $f\in \mathcal{GSR} (\Omega_{D})$.  By Proposition \ref{relation} and \ref{GSM-GSR}, it holds that
$$  D_{\boldsymbol{x} }f(\boldsymbol{x} )=(1-q)  f_{s}^\prime(\boldsymbol{x} ).$$
Together  with Theorem \ref{spherical-derivative-value} (i) for $k=\frac{q-1}{2}$, where $q\in \mathbb{N}$ is  odd,
we obtain
 \[ D_{\boldsymbol{x}}  \Delta_{\bx} ^{\frac{q-1}{2}} f(\boldsymbol{x} )= \Delta_{\bx} ^{\frac{q-1}{2}} D_{\boldsymbol{x} }f(\boldsymbol{x} )=(1-q) \Delta_{\bx} ^{\frac{q-1}{2}} f_{s}^\prime(\boldsymbol{x} )=0.\]
(ii) From the conclusion in (i), we get
\[\Delta_{\bx}^{\frac{q+1}{2}}f(\boldsymbol{x} )=\Delta_{\bx}  \Delta_{\bx} ^{\frac{q-1}{2}}f(\boldsymbol{x} )=\overline{D}_{\boldsymbol{x} }D_{\boldsymbol{x}} \Delta_{\bx}^{\frac{q-1}{2}}f(\boldsymbol{x} )=0.\]
Therefore, every generalized partial-slice regular function $f$ is polyharmonic of degree  $\frac{q+1}{2}$.
The proof is complete.
 \end{proof}
Now we consider the second   enhancement of    the Fueter-Sce theorem for generalized partial-slice  functions satisfying a Vekua-type system, which is closely related to   eigenvalue problems.  Interested readers can refer to \cite{Xu-Zhenyuan} for eigenvalue problems  within the classical monogenic functions over Clifford  algebras and \cite{Krau}   within  slice-regular functions over quaternions.   The following result can be viewed as  a high dimensional   analogue of the Fueter's theorem for one class of pseudoanalytic functions in \cite[Theorem 1.2]{Lian}.
Note that in  the following theorem the function $f$ under   consideration does not depend on the first   variable  $\bx_0$, whence   we write it as $f(\underline{\bx})$.
\begin{theorem}\label{Futer-Sce-Vekua}
Let   $\underline{D}\subseteq \mathbb{R}^{p+1}$ be  a domain, invariant under the reflection  of the $(p+1)$-th variable,
 and $f(\underline{\bx})=F_1(\underline{\bx}')+\underline{\omega} F_2(\underline{\bx}') \in {\mathcal{GS}}^{1}(\Omega_{\underline{D}})$. Given a   constant $\lambda\in \mathbb{R}$, if $F_1, F_2$ satisfy  the Vekua-type system
 \begin{eqnarray}\label{C-R-Vekua}
 \left\{
\begin{array}{ll}
D_{\underline{\bx}_p}  F_1(\underline{\bx}')- \partial_{r} F_2(\underline{\bx}')=\lambda  F_1(\underline{\bx}'),
\\
 - D_{\underline{\bx}_p}  F_2(\underline{\bx}')+ \partial_{r} F_1(\underline{\bx}')= \lambda  F_2(\underline{\bx}'),
\end{array}  \quad \underline{\bx}'=(\underline{\bx}_p,r) \in \underline{D},
\right.
\end{eqnarray}
then
\begin{equation}\label{Fueter-Sce-eigenvalue}
(D_{\underline{\boldsymbol{x}}} -\lambda) (\Delta_{\underline{\bx}}+\lambda^{2}) ^{\frac{q-1}{2}} f(\underline{\boldsymbol{x}} )=0.
\end{equation}
\end{theorem}
\begin{proof}
Let $f(\underline{\bx})=F_1(\underline{\bx}')+\underline{\omega} F_2(\underline{\bx}') \in {\mathcal{GS}}^{1}(\Omega_{\underline{D}})$.  By direct computations,  (\ref{C-R-Vekua}) holds  if and only
if $f(\underline{\bx})$ satisfies the eigenvalue problem
\begin{equation*}\label{eigenvalue}
(D_{\underline{\bx}_p}+\underline{\omega}\partial_{r}) f(\underline{\bx}_p+r\underline{\omega})=
\lambda f(\underline{\bx}_p+r\underline{\omega}),\quad \underline{\bx}_p+r\underline{\omega} \in \Omega_{\underline{D}},
\end{equation*} for all $\underline{\omega}\in \mathbb{S}$.
Set $h(\bx)=e^{-\lambda \bx_0}f(\underline{\bx})$. Then it holds that
$$(D_{\bx_p}+\underline{\omega}\partial_{r}) h(\bx_p+r\underline{\omega})=e^{-\lambda \bx_0}(D_{\underline{\bx}_p}+\underline{\omega}\partial_{r}- \lambda) f(\underline{\bx}_p+r\underline{\omega})=0, $$
which says that the function $h$ can be viewed slice-by-slice to be monogenic   in  $\mathbb{R}\times \underline{D}\subseteq \mathbb{R}^{p+2}$.
In particular, $h$ is real analytic on $\mathbb{R}\times \underline{D}$ and  hence  $f$  is real analytic on $\Omega_{\underline{D}}$, which guarantees the legitimacy of higher differentiation in (\ref{Fueter-Sce-eigenvalue}).

Furthermore,  (\ref{C-R-Vekua})    implies that $F_1, F_2$ satisfy both the Helmholtz equation, respectively,
\begin{equation*}\label{Helmholtz-equation-Vekua-cons}
(\Delta_{\underline{\bx}'}+\lambda^{2})F_j(\underline{\bx}') =0,\quad j=1,2.\end{equation*}
By using the  method  as in Theorem \ref{Fueter-Sce-1}, one can obtain
$$(D_{\underline{\boldsymbol{x}}} -\lambda) (\Delta_{\underline{\bx}}+\lambda^{2}) ^{\frac{q-1}{2}} f(\underline{\boldsymbol{x}} )=(1-q) (\Delta_{\underline{\bx}}+\lambda^{2}) ^{\frac{q-1}{2}} f_{s}^\prime(\underline{\boldsymbol{x}} )=0,$$
which completes the proof.
\end{proof}

\begin{remark}
 Compared with  \cite[Theorem 1.2]{Lian}, our result in  Theorem \ref{Futer-Sce-Vekua} does not restrict the associated functions $F_1$ and $F_2$ to be  real-valued, but  general $\mathbb{R}_{p+q}$-valued.
 \end{remark}

Now, let us go back to the main line of establishing the Almansi-type decomposition for generalized partial-slice
regular functions.  To formulate the desired Almansi-type theorem, we need   to reinterpret the notion of  \textit{zonal function with pole $\eta$}. Specially, choosing $\eta=(1,0,\ldots,0)\in \mathbb{R}^{n+1}$ in (\ref{zonal}), viewed as the paravector $1$,  the orthogonal transformation  $ T\in O(n+1)$  such that $T\eta=\eta$ always takes the form of $\begin{pmatrix} 1 \ \ \ \ \\ \ \ \ T_{n}\end{pmatrix}$ with $ T_{n}\in O(n)$. Hence, the zonal function $f$ with pole $1$, defined in $\mathbb{R}^{n+1}$, satisfying
$$f(x_{0}+\underline{x})=f(x_{0},|\underline{x}|),$$
has axial symmetry with respect to the real axis. Equivalently,
$$f(x_{0}+|\underline{x}|I)=f(x_{0} +|\underline{x}|J), \quad \forall \ I, J \in \mathbb{S}^{n-1}.$$
 Based on this observation,  we call  the function $f$ is symmetric with respect to the real space $\mathbb{R}^{p+1}$,
if  $$f(\bx_{p}+|\underline{\bx}_{q}|I)=f(\bx_{p}+|\underline{\bx}_{q}|J), \quad \forall \ I, J \in \mathbb{S}.$$
i.e.
$$f(\bx_{p}+\underline{\bx}_{q})=f(\bx_{p},|\underline{\bx}_{q}|),$$
or
$$f \circ T = f,\quad \forall\ T=\begin{pmatrix} E_{p+1} \ \ \ \   \\ \ \ \ \ \  \ \  T_{q}\end{pmatrix},$$
 where $E_{p+1}$ is the unit matrix in $\mathbb{R}^{p+1}$ and $ T_{q}\in O(q)$.

Besides, we need the following    lemma.
\begin{lemma}\label{poly-harmonic-lemma}
Let $m\in \mathbb{N}$. For any   $C^{2m}$ function  $h: \mathbb{R}^{p+q+1} \rightarrow \mathbb{R}_{p+q}$, we have
$$\Delta_{\bx}^{m}(\overline{\bx}_p h(\bx))=2m \Delta_{\bx}^{m-1}   \overline{D}_{\bx_p}h(\bx)+   \overline{\bx}_p  (\Delta_{\bx}^{m}h(\bx)).$$
\end{lemma}
\begin{proof}Observing that
$$\Delta_{\bx} (\overline{\bx}_p h(\bx))= 2 \overline{D}_{\bx_p}h(\bx)+\overline{\bx}_p  (\Delta_{\bx}h(\bx)), $$
we obtain
 \begin{eqnarray*}
\Delta_{\bx}^{m}(\overline{\bx}_p h(\bx))
&=&\Delta_{\bx}^{m-1} (2 \overline{D}_{\bx_p}h(\bx)+\overline{\bx}_p  (\Delta_{\bx}h(\bx)) )\\
&=& 2 \Delta_{\bx}^{m-1}   \overline{D}_{\bx_p}h(\bx)+ \Delta_{\bx}^{m-1}( \overline{\bx}_p  (\Delta_{\bx}h(\bx)) )
 \\
&=& 2 \Delta_{\bx}^{m-1}   \overline{D}_{\bx_p}h(\bx)+ 2 \Delta_{\bx}^{m-2}   \overline{D}_{\bx_p} \Delta_{\bx} h(\bx) + \Delta_{\bx}^{m-2}( \overline{\bx}_p ( \Delta_{\bx}^{2}h(\bx)))\\
&=&4 \Delta_{\bx}^{m-1}   \overline{D}_{\bx_p}h(\bx)+ \Delta_{\bx}^{m-2}( \overline{\bx}_p  (\Delta_{\bx}^{2}h(\bx)))\\
& & \cdots\\
&=& 2m \Delta_{\bx}^{m-1}   \overline{D}_{\bx_p}h(\bx)+   \overline{\bx}_p ( \Delta_{\bx}^{m}h(\bx)),
\end{eqnarray*}
where the fourth equality follows from the commutative relationship  $\overline{D}_{\bx_p} \Delta_{\bx}=\Delta_{\bx}\overline{D}_{\bx_p}$.
The proof is complete.
\end{proof}

Now we can  formulate  main result as follows.
\begin{theorem} \label{Almansi-theorem-Main}
Let $q=2m+1\geq3$ be odd and $D\subseteq\mathbb{R}^{p+2}$ be  a domain. If $f:\Omega_D \longrightarrow \mathbb {R}_{p+q} $ is a generalized partial-slice regular function, then there exist two unique    polyharmonic functions of degree $m$ $A$ and $B$ in  $\Omega_D$, symmetric with respect to the real space $\mathbb{R}^{p+1}$, such that
 $$f(\bx)=A(\bx)- \overline{\bx}B(\bx).$$
Conversely, if $A$ and $ B$  are $\mathbb{R}_{p+q}$-valued functions of $C^{1}$ class in $\Omega_D$,  symmetric with respect to the real space $\mathbb{R}^{p+1}$, then $g(x):=A(\bx)- \overline{\bx}B(\bx) \in \mathcal{GS}(\Omega_D)$. Furthermore,  the function $g$ is generalized partial-slice regular   if and only if $A$ and $B$ satisfy the system of equations
 \begin{eqnarray}\label{CR-2}
 \left\{
\begin{array}{ll}
D_{\bx_p} A(\bx')-\sum_{i=0}^{p}e_i\overline{\bx}_p\partial_{x_i}B(\bx')-r\partial_{r}B(\bx')=(p+2)B(\bx'),
\\
 \partial_{r}A(\bx') -   \overline{\bx}_p \partial_{r} B(\bx')+r\overline{D}_{\bx_p} B(\bx')=0.
\end{array}
\right.
\end{eqnarray}
\end{theorem}
\begin{proof}
In view of (\ref{spherical-value-derivative}), we have the following decomposition
$$f(\bx)=f^{\circ}_{s}(\bx)+ \underline{\bx}_{q} f_{s}^\prime(\bx)=A(\bx)- \overline{\bx}B(\bx),$$
 where the generalized partial-slice  functions $A$ and $B$ given by
 $$A(\bx)=f^{\circ}_{s}(\bx)+\overline{\bx}_p f_{s}^\prime(\bx), \ B(\bx)= f_{s}^\prime(\bx),$$
depend only on the variable $\bx'=(\bx_p,r)$ and are even w.r.t. $ r$. That is to say functions $A$ and $B$ are symmetric with respect to the real space $\mathbb{R}^{p+1}$.

Now let us show $A$ and $B$ are polyharmonic functions of degree $m$  in $\Omega_D$.
By Lemma \ref{poly-harmonic-lemma}, we have
$$\Delta_{\bx}^{m} A(\bx)=\Delta_{\bx}^{m} f^{\circ}_{s}(\bx)+2m \Delta_{\bx}^{m-1}   \overline{D}_{\bx_p}f_{s}^\prime(\bx)+   \overline{\bx}_p  \Delta_{\bx}^{m}f_{s}^\prime(\bx).$$
In view of Theorem \ref{spherical-derivative-value}  (i) and  (iv)
$$ \Delta_{\bx}^{m}f_{s}^\prime(\bx)=0, \quad \Delta_{\bx}  f^{\circ}_{s}(\bx)= -2m   \overline{D}_{\bx_p}f_{s}^\prime(\bx),$$
we get the desired results $\Delta_{\bx}^{m} A(\bx)=\Delta_{\bx}^{m} B(\bx)=0$.

To prove the uniqueness, consider the function
$$f(\bx)=A(\bx)- \overline{\bx}B(\bx), \quad \bx\in \Omega_D,$$
where $A$ and $B$ are   symmetric w.r.t. the real   space $\mathbb{R}^{p+1}$.\\
In particular,
$$A(\bx)= A(\bx_{\diamond}),\ B(\bx)=B(\bx_{\diamond}),\quad \bx=\bx_p+\underline{\bx}_q,\ \bx_{\diamond}=\bx_p-\underline{\bx}_q.$$
Hence, by definition, it holds that
$$f_{s}^\prime(\bx)=\frac{1}{2}\underline{\bx}_{q}^{-1}\big( A(\bx)- \overline{\bx}B(\bx)-(A(\bx_{\diamond})-
\overline{\bx_{\diamond}}B(\bx_{\diamond}) ) \big) =B(\bx),$$
and then
$$ A(\bx)= f(\bx)+ \overline{\bx}B(\bx)=f(\bx)+\overline{\bx}f_{s}^\prime(\bx).$$
 Therefore $A$ and $B$ are uniquely determined by $f$.

Conversely, suppose  that $g(x)=A(\bx)- \overline{\bx}B(\bx)$ where $A$ and $ B$  are $\mathbb{R}_{p+q}$-valued functions of $C^{1}$ class in $\Omega_D$,  symmetric with respect to the real space $\mathbb{R}^{p+1}$. Given $\bx_p+r\underline{\omega}   \in \Omega_{D}$ and $\bx'=(\bx_p,r) \in D$, then we can write $g$ as
 $$g(\bx)=G_{1}(\bx')+ \underline{\omega} G_{2}(\bx'),$$
where
$$G_{1}(\bx')=A(\bx)- \overline{\bx}_{p}B(\bx), \ G_{2}(\bx')= r B(\bx).$$
It  follows immediately from the symmetry properties of $A$  and $B$ that $(G_1,G_2)$ is an even-odd pair which induces the generalized partial-slice function $g$. Finally, $(G_1,G_2)$ satisfies the generalized Cauchy-Riemann equations  (\ref{C-R}) if and only if (\ref{CR-2}) holds by direct computations, which finishes the proof.
\end{proof}

From  Theorems   \ref{Almansi-decomposition-harmonic}  and \ref{Almansi-theorem-Main}, one may  generalize  Theorem \ref{Almansi-theorem-Perotti-2} into the setting of  generalized  partial-slice regular functions as follows.
\begin{theorem} \label{Almansi-theorem-Perotti-Xu-2}
 Let $q=2m+1\geq3$ be odd, $\Omega_{D}$ be  a star-like domain in $\mathbb{R}^{p+q+1}$ with centre $0$. If $f$ is  a generalized partial-slice regular function in  $\Omega_D$, then  there exist  unique
 functions    $g_{0},g_{1}, \ldots,g_{m-1}\in Ker (D_{\bx} \Delta_{\bx} )$,   symmetric with respect to the real space $\mathbb{R}^{p+1}$,  such that
  $$f(\bx)=g_{0}(\bx)+|\bx|^{2}g_{1}(\bx)+ \cdots+|\bx|^{2m-2}g_{m-1}(\bx), \quad \bx\in \Omega_D.$$
\end{theorem}
\begin{proof}We shall follow   the argument of Theorem \ref{Almansi-theorem-Perotti-2} given in  \cite{Perotti-21}.
By Theorem \ref{Almansi-theorem-Main},   there exist two unique    polyharmonic functions of degree $m$ $A$ and $B$ in  $\Omega_D$, symmetric with respect to the real space $\mathbb{R}^{p+1}$, such that
 $$f(\bx)=A(\bx)- \overline{\bx}B(\bx).$$
By Theorem \ref{Almansi-decomposition-harmonic},  there exist unique    $\mathbb{R}_{p+q}$-valued harmonic functions
$u_0,\ldots,u_{m-1}$ and $v_0,\ldots,v_{m-1}$  in  $\Omega_D$ such that
\begin{equation}\label{A}
A(\bx)=\sum_{k=0}^{m-1}|\bx|^{k}u_k(\bx), \quad B(\bx)=\sum_{k=0}^{m-1}|\bx|^{k}v_k(\bx).\end{equation}
Hence,
 $$f(\bx)=\sum_{k=0}^{m-1}|\bx|^{k}g_k(\bx), \quad g_k(\bx)=u_k(\bx)- \overline{\bx}v_k(\bx),$$
 where  each function $g_k$ is unique. Furthermore,   each function $g_k$ is symmetric with respect to the real space $\mathbb{R}^{p+1}$  as the functions $A$ and $B$. To see this, consider the orthogonal
transformation $ T=\begin{pmatrix} E_{p+1} \ \ \ \   \\ \ \ \ \ \  \ \  T_{q}\end{pmatrix},$
with $E_{p+1}$ being the unit matrix in $\mathbb{R}^{p+1}$ and $ T_{q}\in O(q)$, then we get
$$A \circ T(\bx)=\sum_{k=0}^{m-1}|T\bx|^{k}u_k(T(\bx))=\sum_{k=0}^{m-1}|\bx|^{k}u_k(T(\bx)).$$
In view of the condition  $A \circ T=A$ and  the uniqueness of the Almansi decomposition in (\ref{A}), we obtain $u_k\circ  T=u_k$ for each $k$. The same holds for $v_k$.  Finally, let us  prove  $D_{\bx} \Delta_{\bx}g_k(\bx)=0$  as
$$D_{\bx} \Delta_{\bx}g_k(\bx)=-D_{\bx} \Delta_{\bx} (\overline{\bx}v_k(\bx) )=-2D_{\bx} \overline{D}_{\bx} v_k(\bx) =
-2\Delta_{\bx}v_k(\bx)=0.$$
The proof is complete.
\end{proof}


\textbf{Declarations}
\\
\textbf{Author contributions}
All authors have contributed equally to all aspects of this manuscript and have reviewed its final draft.
\\
\textbf{Conflict of interest}
There is no financial or non-financial interests that are directly or indirectly related to the work submitted for publication.
\\
\textbf{Data availability}
Data sharing is not applicable to this article as no datasets were generated  during the current study.
\bigskip
\\
\textbf{Acknowledgements}
 The authors would like to thank Professor Irene Sabadini for her useful comments.





\vskip 10mm
\end{document}